\documentclass[12pt]{article}

\usepackage{amssymb,a4}
\usepackage{amsmath,amsfonts,amssymb,amsthm}

\newcommand{\Z}{{\mathbb Z}}

\newcommand{\C}{{\mathbb C}}

\newcommand{\N}{{\mathbb N}}

\def\<{\langle}
\def\>{\rangle}
\def\a{\alpha}
\def\b{\beta}
\def\d{\delta}

\newtheorem{thm}{Theorem}[section]
\newtheorem{prop}[thm]{Proposition}
\newtheorem{lem}[thm]{Lemma}

\newtheorem{rmk}[thm]{Remark}
\newtheorem{definition}[thm]{Definition}

\begin{document}

\makeatletter \@addtoreset{equation}{section}
\def\theequation{\thesection.\arabic{equation}}
\makeatother \makeatletter

\begin{center}
{\Large \bf   Verma modules for rank two Heisenberg-Virasoro algebra}
\end{center}

\begin{center}
{Zhiqiang Li,
Shaobin Tan$^{a}$\footnote{Partially supported by China NSF grant (Nos.11471268, 11531004).}\\
$\mbox{}^{a}$School of Mathematical Sciences, Xiamen University,
Xiamen 361005, China}
\end{center}

\begin{abstract}

Let $\preceq$ be a compatible total order on the additive group $\Z^2$, and $L$ be the rank two Heisenberg-Virasoro algebra. For any $\mathbf{c}=(c_1,c_2,c_3,c_4)\in\C^4$,
we define $\Z^2$-graded Verma module $M(\mathbf{c}, \preceq)$ for the Lie algebra $L$.  A necessary and sufficient condition for the Verma module  $M(\mathbf{c}, \preceq)$ to be  irreducible is provided.
Moreover, the maximal $\Z^2$-graded submodules of the Verma module $M(\mathbf{c}, \preceq)$ are characterized when $M(\mathbf{c}, \preceq)$ is  reducible.
\end{abstract}

\section{Introduction}

Verma modules are one of the most important research object in the study of the representation theory of infinite dimensional
Lie algebras. There are some known results about Verma modules for several graded Lie algebras closely related to the rank two Heisenberg-Virasoro algebra,
such as the twisted Heisenberg-Virasoro algebra (\cite{ACKP,B}),
the generalized Heisenberg-Virasoro algebra (\cite{SJS}),
the generalized Virasoro algebra (\cite{HWZ}) and the Virasoro-like algebra (\cite{WZ}).
In this paper, we study the Verma modules for the rank two Heisenberg-Virasoro algebra $L$.

The Lie algebra $L$, studied in this paper, can be viewed as a generalization of the twisted Heisenberg-Virasoro algebra
(see \cite{XLT,TWX} for details). The structure and representation theory for
the twisted Heisenberg-Virasoro algebra have been studied in  \cite{ACKP,B,JJ,SJ}. The Verma module structure has been characterized
in \cite{ACKP} and \cite{B}. In this paper, we consider the Verma modules for the rank two Heisenberg-Virasoro algebra $L$.
 The derivations, automorphism group and central extension for $L$ were studied in \cite{XLT}. In \cite{JL}, the authors classified
all irreducible $\Z^2$-graded $L$-module under the assumption that the action of the torus is associative.
Furthermore, in \cite{GL2}, it was proved that this assumption is equivalent to
that the $\Z^2$-graded $L$-module is uniformly bounded and the torus acts non-trivially.
The universal Whittaker modules for $L$ were discussed in \cite{TWX}, where the irreducibility of the universal Whittaker modules was determined.
 Recently, certain relationship between the restricted module categories for the Lie algebra $L$ and  module categories of
the corresponding vertex operator algebra over $L$ has been established in \cite{GW}.

The paper is organized as follows. In Section $2$, we first recall the definition of rank two Heisenberg-Virasoro algebra $L$, and the notion of
compatible total order $\preceq$ (see \cite{HWZ,WZ}) over an additive group. We will see that the Lie algebra $L$ (see Definition $2.1$) has a natural $\Z^{2}$-gradation
$$L=\oplus_{\alpha\in\Z^{2}}L_{\alpha},\ L_{0}=\sum\limits_{i=1}^{4}\C K_{i},\
L_{\alpha}=\C t^{\alpha}\oplus\C E(\alpha),$$for $\alpha\in \Z^{2}\setminus\{(0,0)\}$.
We fix a compatible total order $\preceq$ on $\Z^{2}$, then we have a triangular decomposition for the Lie algebra $L$, and consequently we are able to
define a $\Z^{2}$-graded highest weight modules (called Verma modules) $M(\mathbf{c}, \preceq)$, where
$\mathbf{c}=(c_{1}, c_{2}, c_{3}, c_{4})\in \C^{4}$. In Section $3$, we give a necessary and sufficient condition
for the Verma module $M(\mathbf{c}, \preceq)$ to be an irreducible $L$-module.  We remark that our method is different from that given in [WZ] for Virasoro-like algebra. Finally, we determine the maximal $\Z^2$-graded submodules for the Verma module $M(\mathbf{c}, \preceq)$ whenever $M(\mathbf{c}, \preceq)$ is reducible. The arguments are divided into two cases according to the compatible total order $\preceq$ on the group $\Z^2$ being dense or discrete.

We denote by $\C$, $\Z$, $\N$ and
$\Z_{+}$ the sets of all complex numbers, integers, positive integers and
nonnegative integers, respectively. All vector spaces mentioned in
this paper are over the complex field $\C$. For any Lie algebra $\mathfrak{g}$, the universal enveloping algebra of $\mathfrak{g}$
is denoted by $\mathcal{U}(\mathfrak{g})$.

\section{Verma modules and dense orders}

If $G$ is an additive group, we say a total order $\preceq$ on $G$ is {\it compatible} if
$\a\preceq\b$ implies $\a+\gamma\preceq \b+\gamma$ for every $\gamma\in G$. By $\a\prec\b$ we mean that $\a\preceq\b$ and $\a\not=\b$, and by $\a\succeq\b,\ \a\succ\b$ we mean that $\b\preceq\a,\ \b\prec\a$ respectively. Moreover,
 For any
$\underline{\a}=(\a_{1}, \cdots, \a_{r})$ and $\underline{\b}=(\b_{1}, \cdots, \b_{r})$,
where $\a_{i}$, $\b_{i}\in G$, $1\leq i\leq r$, $r\in\N$, we define
$\underline{\a}\succ\underline{\b}$ if and only if there exists some $1\leq s\leq r$ such that $\a_{s}\succ\b_{s}$ and
$\a_{i}=\b_{i}$ for any $i>s$. Similarly, one defines $\underline{\a}\prec\underline{\b}$ if and only if $\underline{\b}\succ\underline{\a}$. More generally, for any $\underline{\alpha}
=(\alpha_{1}, \alpha_{2},\cdots, \alpha_{n})$ and $\underline{\beta}=
(\beta_{1}, \beta_{2},\cdots, \beta_{m})$ with $m>n$, we define $\underline{\beta}\succ\underline{\alpha}$ (respectively $\underline{\beta}\prec\underline{\alpha}$) if and only if $
\underline{\beta}\succ(0, \underline{\alpha})$ (respectively $\underline{\beta}\prec(0, \underline{\alpha})$),
where the number of $0$'s (the identity element in the additive group $G$) is equal to $m-n$. Let
$$
G_{+}=\{\a\in G\mid\a\succ 0\},\ \ G_{-}=\{\a\in G\mid\a\prec 0\}.
$$
Then $G=G_{+}\cup\{0\}\cup G_{-}$.  For any $\a\in G_{+}$, let $B(\a)$ denote the set $\{\b\in G_+\mid \b\prec\a\}$
and let $D(\a)$ denote the set $\{\b\in G_{+}\mid m\b=\a$, for some $m\in \N\}$.
Clearly the set $D(\a)$ is a finite set, and we denote the minimal element in $D(\a)$ with
respect to the order $\preceq$ by $\mu(\a)$.
We say that the order $\preceq$ is \textbf{dense} if for any
$\alpha\in G_{+}$ there exists some $\beta\in G_{+}$ such that $\beta\prec \alpha$ and that
$\preceq$ is \textbf{discrete} if there exists a smallest positive element, say $\epsilon$, in $G$, i.e.,
$\epsilon\prec \alpha$ for any other $\alpha\in G_{+}$.
Obviously the order $\preceq$ is dense if and only if $B(\a)$ is an infinite set for any $\a\in G_{+}$, and the order $\preceq$
is discrete if and only if there exists some $\epsilon\in G_{+}$, such that $B(\epsilon)=\emptyset$.

From now on, we let $G$ be the additive group $\Z^{2}$, and  fix a compatible total order $\preceq$ on the group $G$. Whenever the order $\preceq$ is discrete,
we denote the smallest positive element in $G$ by $\epsilon=(\epsilon(1),\epsilon(2))\in\Z^2$.

For any element $\alpha\in G$, we denote
$\alpha=(\alpha(1),\alpha(2))$, and we set
$\text{det}{\beta \choose \alpha}=\beta(1)\alpha(2)-\alpha(1)\beta(2)$ for $\alpha$, $\beta\in G$. Now we recall the definition of the rank two Heisenberg-Virasoro algebra $L$.

\begin{definition}\label{defHVart}
{\em The {\em rank two Heisenberg-Virasoro algebra} is a Lie algebra spanned by elements of the form
\[\{t^{\alpha}, E(\alpha), K_{i}\mid\alpha\in G\setminus\{(0,0)\},\ \
i=1, 2, 3, 4\}\]} with relations defined by
\[[t^{\alpha}, t^{\beta}]=0,\ [K_{i}, L]=0, \ \ i=1, 2, 3, 4;\]
\[[t^{\alpha}, E(\beta)]=\text{det}{\beta \choose \alpha}
t^{\alpha+\beta}+\delta_{\alpha+\beta, 0}h(\alpha);\]
\[[E(\alpha), E(\beta)]=\text{det}{\beta \choose \alpha}
E(\alpha+\beta)+\delta_{\alpha+\beta, 0}f(\alpha),\]
where $h(\alpha)=\alpha(1)K_{1}+\alpha(2)K_{2}$, $f(\alpha)=\alpha(1)K_{3}+\alpha(2)K_{4}$.
\end{definition}

We notice that the Virasoro-like algebra ([GL1], [WZ]) is a Lie subalgebra of the rank two Heisenberg-Virasoro algebra $L$ spanned by the subset
$\{E(\alpha),\ K_{3},\ K_{4}\mid \alpha\in G\setminus\{(0, 0)\}\}$.

Set $L_{\pm}=\sum\limits_{\pm\alpha\in G_{+}} (\C t^{\alpha}\oplus \C E(\alpha))$
and $L_{0}=\sum\limits_{i=1}^{4}\C K_{i}$.
Then we have the following triangular decomposition of $L$:
$$L=L_{-}\oplus L_{0}\oplus L_{+}.$$
Set $L^{t}_{\pm}=\sum_{\pm\alpha\in G_{+}}\C t^{\alpha}$, $L^{E}_{\pm}=\sum_{\pm\alpha\in G_{+}}\C E(\alpha)$,
then, $L_{\pm}=L^{t}_{\pm}\oplus L^{E}_{\pm}$.
By PBW theorem, the universal enveloping algebra $\mathcal{U}=\mathcal{U}(L)$ of $L$ can be factored as follows
$$\mathcal{U}(L)=\mathcal{U}(L_{-})\mathcal{U}(L_{0}) \mathcal{U}(L_{+}).$$
Now we define the Verma modules for $L$ with respect to the above decomposition. Let
$\C v_L$ be the one dimensional module of the Lie subalgebra $L_{0}\oplus L_{+}$, which is defined
by $L_{+}.v_L=0$ and $K_{i}.v_L=c_{i}v_L$, where $c_{i}\in \C$ for $i=1, 2, 3, 4$.
Denote $\mathbf{c}=(c_{1}, c_{2}, c_{3}, c_{4})\in \C^{4}$. Then we define the
Verma module for $L$ \begin{eqnarray}M(\mathbf{c}, \preceq)=\mathcal{U}(L)\otimes_{\mathcal{U}(L_{0}\oplus L_{+})}\C v_L.\end{eqnarray}

From PBW theorem, we know $M(\mathbf{c}, \preceq)\simeq\mathcal{U}(L_-)\otimes\C v_L$ as vector spaces.
For any $\a\in G_{-}$, let $M(\mathbf{c}, \preceq)_{\a}$ denote the vector space
spanned by elements of the form
$$t^{-\alpha_{1}}t^{-\alpha_{2}}\cdots t^{-\alpha_{m}}E(-\beta_{1})E(-\beta_{2})\cdots E(-\beta_{n})v_L,$$
where $m, n\in \Z_{+}$, $\alpha_{i}, \beta_{j}\in G$,
$0\prec\alpha_{1}\preceq\alpha_{2}\preceq\cdots \preceq \alpha_{m}$, $0\prec\beta_{1}\preceq\beta_{2}\preceq
\cdots \preceq\beta_{n}$ and such that $\sum\limits_{i=1}^{m}\alpha_{i}+\sum\limits_{j=1}^{n}\beta_{j}=-\a$.
Let $M(\mathbf{c}, \preceq)_{0}=\C v_L$. It is clear that
$M(\mathbf{c}, \preceq)=\sum\limits_{\a\preceq0}M(\mathbf{c}, \preceq)_{\a}$ is a $\Z^2$-graded $L$-module.
Now we give two lemmas which will be used in Section $3$.
\begin{lem}
Suppose that the compatible total order $\preceq$ on $G$ is dense. Then

(1). For any $\a$, $\b\in G_{+}$, one can choose an element $\a_{1}\in B(\a)$ such that
$\text{det}{\b \choose \gamma}\neq 0$ for all $\gamma\in B(\a_{1})$.

(2). The order $\preceq$ on $G$ is standard in the sense that for any $\a$, $\b\in G_{+}$, there
exists a positive integer $n$ such that $n\a\succ\b$.

(3). For any $0\prec\a\prec\b$, there exists some $\gamma\in G$ such that $\a\prec\gamma\prec\b$ and $\text{det}{\b \choose\gamma}\neq0$.
\end{lem}
\begin{proof}
The statements $(1)$ and $(2)$ are due to the Lemma $2.1$ in \cite{WZ}. We only need to prove $(3)$. From $(2)$,
we know that there exists a unique $n\in \Z_{+}$ such that $n\mu(\b)\prec\a\preceq (n+1)\mu(\b)$.
If $\a=(n+1)\mu(\b)$, then $\a\prec (n+2)\mu(\b)\preceq\b$. Now for any $\a\prec\gamma\prec (n+2)\mu(\b)$, we have
$\a\prec\gamma\prec\b$ and $\text{det}{\b\choose\gamma}\neq0$. If $n\mu(\b)\prec\a\prec (n+1)\mu(\b)$, then it is obvious that
$\b\succeq(n+1)\mu(\b)$. Thus for any $\gamma\in G$ such that $\a\prec\gamma\prec (n+1)\mu(\b)$, we have that
$\a\prec\gamma\prec\b$ and $\text{det}{\b\choose\gamma}\neq0$.
\end{proof}

\begin{lem}
Suppose that the compatible total order $\preceq$ on $G$ is discrete. Then

(1). There exists $\epsilon^{\prime}\succ 0$ such that $\{\epsilon,\ \epsilon^{\prime}\}$
forms a $\Z$-basis of $G$.

(2). If $0\preceq \a\prec\b$ such that $(\b-\a)\notin \Z \epsilon$, then $\b\succ \a+n\epsilon$ for any $n\in \N$.
\end{lem}

\begin{proof}
The statement $(1)$ is from the Lemma $2.1$ in \cite{GL1}. Now we prove (2). Since $\b-\a\succ0$ and $(\b-\a)\notin \Z \epsilon$, we get that $\b-\a\succ\epsilon$.
Thus $\b-\a-\epsilon\succ0$. Together with $(\b-\a)\notin \Z \epsilon$, we get $\b-\a\succ2\epsilon$.
By induction, we get $\b\succ \a+n\epsilon$ for any $n\in \N$.
\end{proof}

\begin{rmk}

(1). When the compatible total order $\preceq$ is dense, then dim $M(\mathbf{c}, \preceq)_{\a}=\infty$ for any $\a\in G_{-}$.

(2). When the compatible total order $\preceq$ is discrete, then dim $M(\mathbf{c}, \preceq)_{\a}=\infty$ for any $\a\in G_{-}\setminus\Z\epsilon$.\\
In fact, since the set $B(-\alpha)$ is infinite if the order $\preceq$ is dense, then for any $\gamma\in G_+$ such that
$\gamma\prec\alpha$, we have $t^{-\gamma}E(\alpha+\gamma)v_L\in M(\mathbf{c}, \preceq)_{\a}$. And When the order $\preceq$ is discrete,
we have $-\alpha-n\epsilon\in G_+$ and $t^{-n\epsilon}E(\alpha+n\epsilon)v_L\in M(\mathbf{c}, \preceq)_{\a}$ for $n\in\N$.

\end{rmk}

Now if the order is discrete, we assume $\{\epsilon,\ \epsilon^{\prime}\}$ forms a $\Z$-basis of $G$, where $\epsilon^{\prime}\succ 0$.
For later convenience and use, we adopt the symbol $\alpha\succ\Z \epsilon$ to mean that $\alpha\in G_{+}\setminus\Z \epsilon$, and we set that $T_{\epsilon}=\sum_{k\in \N}\C t^{-k\epsilon}$,
$E_{\epsilon}=\sum_{k\in \N}\C E(-k\epsilon)$.

\section{Irreducibility of Verma modules}
In this section, we shall determine the necessary and sufficient conditions
for the Verma modules $M(\mathbf{c}, \preceq)$ to be irreducible. And when
the Verma module $M(\mathbf{c}, \preceq)$ is reducible, we get its maximal submodule.

Now we give some definitions that will used later. For $0\neq u\in M(\mathbf{c}, \preceq)$, we set
$$u=\sum\limits_{i=1}^{m}\sum\limits_{j=1}^{n}a_{ij}t^{-\alpha_{i_{1}}}
t^{-\alpha_{i_{2}}}\cdots t^{-\alpha_{i_{k_{i}}}}E(\beta_{j_{1}})E(\beta_{j_{2}})\cdots E(\beta_{j_{d_{j}}})v_L,$$
where $0\neq a_{ij}\in\C$, $0\prec\alpha_{i_{1}}\preceq\alpha_{i_{2}}\preceq\cdots\preceq\alpha_{i_{k_{i}}}$,
$0\prec\beta_{j_{1}}\preceq\beta_{j_{2}}\preceq\cdots\preceq\beta_{j_{d_{j}}}$ for $1\leq i\leq m$, $1\leq j\leq n$.
We define \[J(u)=\{\alpha_{i_{s}}\mid1\leq i\leq m,\ 1\leq s\leq k_{i}\},\ \ \
r(u)=\text{max}\{d_{j}\mid1\leq j\leq n\}.\]Note that if $0\neq u\in\mathcal{U}(L^{E}_{-})v_L$, $J(u)=\emptyset$;
and if $0\neq u\in\mathcal{U}(L^{t}_{-})v_L$, $r(u)=0$. Suppose $r(u)=n\in \N$. We can write
\begin{eqnarray}u=\sum_{i=1}^{m}x_{i}E(-\beta_{i_{1}})E(-\beta_{i_{2}})\cdots E(-\beta_{i_{n}})v_L+u^{\prime},\end{eqnarray}
where $0\neq x_{i}\in \mathcal{U}(L^{t}_{-})$, $0\prec\beta_{i_{1}}\preceq\beta_{i_{2}}\preceq\cdots\preceq\beta_{i_{n}}$
for $1\leq i\leq m$ and $u^{\prime}\in M(\mathbf{c}, \preceq)$ such that $r(u^{\prime})<n$ or $u^{\prime}=0$.
For the expression in $(3.1)$, set $\underline{\beta_{i}}=(\beta_{i_{1}}, \beta_{i_{2}},\cdots, \beta_{i_{n}})$ and we define \[I(u)=\{\underline{\beta_{i}}\mid1\leq i\leq m\}.\]
The following two lemmas play critical role in the proving of the irreducibility of $M(\mathbf{c}, \preceq)$ under
the condition that the order $\preceq$ on $G$ is dense.

\begin{lem}
When the compatible total order $\preceq$ on $G$ is dense,
for any $u\in M(\mathbf{c}, \preceq)$ such that $r(u)=m>1$, then there is an element
$u^{\prime}\in\mathcal{U}(L)u$, such that $r(u^{\prime})=1$.
\end{lem}

\begin{proof}
Write $u=\bar{u}+u_{0}$, where
\[\bar{u}=\sum_{i=1}^{k}\sum_{j=1}^{d}b_{ij}t^{-\alpha_{i_{1}}}
t^{-\alpha_{i_{2}}}\cdots t^{-\alpha_{i_{s_{i}}}}E(-\beta_{j_{1}})E(-\beta_{j_{2}})\cdots E(-\beta_{j_{m}})v_L,\]
$0\neq b_{ij}\in\C$, $0\prec\alpha_{i_{1}}\preceq\alpha_{i_{2}}\preceq\cdots\preceq\alpha_{i_{s_{i}}}$, $0\prec\beta_{j_{1}}\preceq\beta_{j_{2}}\preceq\cdots\preceq\beta_{j_{m}}$,
for $1\leq i\leq k$, $1\leq j\leq d$, and $u_{0}\in M(\mathbf{c}, \preceq)$ such that $r(u_{0})<m$ or $u_{0}=0$.
We have \[I(u)=\{\underline{\beta_{j}}=(\beta_{j_{1}}, \beta_{j_{2}},\cdots, \beta_{j_{m}})\mid1\leq j\leq d\}.\]
Without loss of generality, we can assume $\underline{\beta_{1}}\succeq\underline{\beta_{2}}\succeq\cdots\succeq\underline{\beta_{d}}$. Let $a$
be the number of the element $\b_{1_1}$ occurring in the sequence $\underline{\beta_{1}}$. Set
$$\gamma=\text{min}\{\beta_{j_{1}}\mid1\leq j\leq d\}.$$
Since $B(\gamma)\cap B(\mu(\beta_{1_{1}}))=B(\text{min}\{\gamma,\ \mu(\beta_{1_{1}})\})$ is an infinite set, we can
always take some $\gamma^{\prime}\in B(\gamma)\cap B(\mu(\beta_{1_{1}}))$ such that
$$\beta_{1_{1}}-\gamma^{\prime}\notin J(\bar{u}),$$
where $J(\bar{u})=\{\alpha_{i_{j}}\mid1\leq i\leq k,\ 1\leq j\leq s_{i}\}$ is a finite set.
Since $\gamma^{\prime}\prec \mu(\beta_{1_{1}})$, $\text{det}{-\beta_{1_{1}} \choose \gamma^{\prime}}\neq0$.
Consider $t_{\gamma^{\prime}}.u$, and we have \[t_{\gamma^{\prime}}.\bar{u}=\sum_{i=1}^{k}a\text{det}{-\beta_{1_{1}} \choose \gamma^{\prime}}b_{i1}t^{-\alpha_{i_{1}}}
t^{-\alpha_{i_{2}}}\cdots t^{-\alpha_{i_{s_{i}}}}t^{-\beta_{1_{1}}+
\gamma^{\prime}}E(-\beta_{1_{2}})\cdots E(-\beta_{1_{m}})v_L+A,\]for some $A\in M(\mathbf{c}, \preceq)$.
Notice that $\beta_{1_{1}}-\gamma^{\prime}\notin J(\bar{u})$
and $\underline{\beta_{1}}\succeq\underline{\beta_{2}}\succeq\cdots\succeq\underline{\beta_{d}}$, one can deduce that
$$\sum_{i=1}^{k}a\text{det}{-\beta_{1_{1}} \choose \gamma^{\prime}}b_{i1}t^{-\alpha_{i_{1}}}
t^{-\alpha_{i_{2}}}\cdots t^{-\alpha_{i_{s_{i}}}}t^{-\beta_{1_{1}}+
\gamma^{\prime}}E(-\beta_{1_{2}})\cdots E(-\beta_{1_{m}})v_L$$
is linearly independent with $A$ if $A\neq0$, i.e., $t_{\gamma^{\prime}}.\bar{u}\neq0$.
It is clear that $r(t_{\gamma^{\prime}}.\bar{u})=m-1$
and $t_{\gamma^{\prime}}.u_{0}=0$, or $r(t_{\gamma^{\prime}}.u_{0})<m-1$. Now
we get that $0\neq t_{\gamma^{\prime}}.u\in\mathcal{U}(L)u$ and
$r(t_{\gamma^{\prime}}.u)=m-1$. Repeating the above process, we know that this
process will terminate after finite steps as $m<\infty$. Then we obtain $u^{\prime}\in\mathcal{U}(L)u$ and such that $r(u^{\prime})=1$.
\end{proof}

\begin{lem}
When the compatible total order $\preceq$ on $G$ is dense, then,
for any $w\in M(\mathbf{c}, \preceq)$ such that $r(w)=1$, there exists an element
$0\neq w^{\prime}\in\mathcal{U}(L)w\cap\mathcal{U}(L^{t}_{-})v_L$.
\end{lem}

\begin{proof}
Suppose that
$$w=\sum\limits_{i=1}^{n}a_{i}t^{-\alpha_{i_{1}}}
t^{-\alpha_{i_{2}}}\cdots t^{-\alpha_{i_{s_{i}}}}E(-\beta_{i})v_L+w_{0},$$
where $a_{i}\neq0$,
$0\prec\alpha_{i_{1}}\preceq\alpha_{i_{2}}\preceq\cdots\preceq\alpha_{i_{s_{i}}}$ for
$1\leq i\leq n$, $0\prec \beta_{n}\prec \beta_{n-1}\prec\cdots\prec\beta_{1}$, and $r(w_{0})=0$ or $w_{0}=0$.
From Lemma $2.2(3)$, we can choose some $\alpha\in G_{+}$ such that $\beta_{2}\prec\alpha\prec\beta_{1}$
and $\text{det}{\beta_{1} \choose \alpha}\neq0$. Let $w^{\prime}=t^{\alpha}.w$. Then it is obvious that $$t^{\alpha}.(\sum\limits_{i=2}^{n}a_{i}t^{-\alpha_{i_{1}}}
t^{-\alpha_{i_{2}}}\cdots t^{-\alpha_{i_{s_{i}}}}E(-\beta_{i})v_L+w_{0})=0.$$
Combine this together with the following equality
$$t^{\alpha}.a_{1}t^{-\alpha_{1_{1}}}
t^{-\alpha_{1_{2}}}\cdots t^{-\alpha_{1_{s_{1}}}}E(-\beta_{1})v_L
=\text{det}{-\beta_{1} \choose \alpha}a_{1}t^{-\alpha_{1_{1}}}
t^{-\alpha_{1_{2}}}\cdots t^{-\alpha_{1_{s_{1}}}}t^{-\beta_{1}+\alpha}v_L,$$
we have $0\neq w^{\prime}\in\mathcal{U}(L)w\cap\mathcal{U}(L^{t}_{-})v_L$.
\end{proof}

Now we state the the necessary and sufficient conditions
for the Verma modules $M(\mathbf{c}, \preceq)$ to be irreducible under the assumption that the compatible total order $\preceq$ is dense on $G$.

\begin{thm}
Suppose that the compatible total order $\preceq$ is dense. Then the Verma module $M(\mathbf{c}, \preceq)$
is irreducible if and only if $(c_{1}, c_{2})\neq(0, 0)$.
\end{thm}

\begin{proof}
If $(c_{1}, c_{2})=(0, 0)$, it is clear that the submodule, generated by the set
$$\{t^{-\alpha}v_L\mid \alpha\in G_{+}\}, $$ is contained in $\sum\limits_{\a\prec0}M(\mathbf{c}, \preceq)_{\a}$.
Thus the Verma module $M(\mathbf{c}, \preceq)$ is reducible. Now we suppose that $(c_{1}, c_{2})\not=(0, 0)$.
From Lemma $3.1$ and Lemma $3.2$, for any $0\neq u\in M(\mathbf{c}, \preceq)$, we can choose an element
$0\neq u^{\prime}\in\mathcal{U}(L)u\cap\mathcal{U}(L^{t}_{-})v_L$. Without loss of generality,
we assume that $0\neq u\in\mathcal{U}(L^{t}_{-})v_L$. Let
$$u=\sum\limits_{i=1}^{k}a_{i}t^{-\alpha_{i_{1}}}
t^{-\alpha_{i_{2}}}\cdots t^{-\alpha_{i_{s_{i}}}}v_L,$$
where $0\not=a_{i}\in\C$, $0\prec\alpha_{i_{1}}\preceq\alpha_{i_{2}}\preceq\cdots\preceq\alpha_{i_{s_{i}}}$ for
$1\leq i\leq k$. Let $\underline{\a_{i}}=(\alpha_{i_{1}}, \alpha_{i_{2}}, \cdots, \alpha_{i_{s_{i}}})$ for
$1\leq i\leq k$, and we may assume that $\underline{\a_{1}}\succ \underline{\a_{2}}\succ\cdots \succ\underline{\a_{k}}$.
Let $b$ be the number of the element $\a_{1_{s_1}}$ occurring in the sequence $\underline{\a_{1}}$. Then there exists $1\leq k_1\leq k$ such that the number of the element $\alpha_{1_{s_{1}}}$ occurring in the sequence
$\underline{\a_{i}}$ for $1\leq i\leq k_{1}$ are all equal to $b$, while for the remaining sequences $\underline{\a_{i}}$ for $k_{1}< i\leq k$ the element $\alpha_{1_{s_{1}}}$ occurs strictly less then $b$ times.

Now we divide our argument into two cases as follows.

\textbf{Case $1$}. If $\alpha(1)c_{1}+\alpha(2)c_{2}\neq0$ for all $\alpha\in J(u)$, where
$J(u)=\{\alpha_{i_{j}}\mid1\leq i\leq k,\ 1\leq j\leq s_{i}\}$.
We consider $E(\alpha_{1_{1}})E(\alpha_{1_{2}})\cdots E(\alpha_{1_{s_{1}}}).v_L$.
Since $\underline{\a_{1}}$ is the maximal element in the set
$\{\underline{\a_{i}}\mid1\leq i\leq k\}$, we have
$$E(\alpha_{1_{1}})E(\alpha_{1_{2}})\cdots E(\alpha_{1_{s_{1}}}).\sum\limits_{i=2}^{k}a_{i}t^{-\alpha_{i_{1}}}
t^{-\alpha_{i_{2}}}\cdots t^{-\alpha_{i_{s_{i}}}}v_L=0,$$
and it is straightforward to check that
$$0\neq E(\alpha_{1_{1}})E(\alpha_{1_{2}})\cdots E(\alpha_{1_{s_{1}}}).t^{-\alpha_{1_{1}}}
t^{-\alpha_{1_{2}}}\cdots t^{-\alpha_{1_{s_{1}}}}v_L\in\C v_L.$$
This implies that
$$0\neq E(\alpha_{1_{1}})E(\alpha_{1_{2}})\cdots E(\alpha_{1_{s_{1}}}).u\in\C v_L.$$
Thus we get $v_L\in\mathcal{U}(L)u$.

\vspace{3mm}

\textbf{Case $2$}. If there exists some $\alpha\in J(u)$ such that $\alpha(1)c_{1}+\alpha(2)c_{2}=0$,
we claim that the set \[X:=\{\beta\in G_{+}\mid\beta\preceq \alpha_{1_{s_{1}}},\ \beta(1)c_{1}+\beta(2)c_{2}=0\}\]is finite.
In fact, $\mu(\alpha)$ is the smallest element in the set $X$ and $X\subset\{m\mu(\alpha)\mid m\in\N\}$.
Thus by Lemma $2.2(2)$, we get that $X$ is a finite set.

If $\alpha_{1_{s_{1}}}(1)c_{1}+\alpha_{1_{s_{1}}}(2)c_{2}\neq0$, then we consider
$(E(\alpha_{1_{s_{1}}}))^{b}.u$. It is straightforward to check that
$$(E(\alpha_{1_{s_{1}}}))^{b}.\sum\limits_{i=k_{1}+1}^{k}a_{i}t^{-\alpha_{i_{1}}}
t^{-\alpha_{i_{2}}}\cdots t^{-\alpha_{i_{s_{i}}}}v_L=0,$$
and
$$(E(\alpha_{1_{s_{1}}}))^{b}.u=\sum\limits_{i=1}^{k_{1}}b!
(\alpha_{1_{s_{1}}}(1)c_{1}+\alpha_{1_{s_{1}}}(2)c_{2})^{b}a_{i}t^{-\alpha_{i_{1}}}
t^{-\alpha_{i_{2}}}\cdots t^{-\alpha_{i_{s_{i}-b}}}v_L\neq0,$$
where $0\prec\alpha_{i_{1}}
\preceq\alpha_{i_{2}}\preceq\cdots\preceq\alpha_{i_{s_{i}-b}}\prec\alpha_{1_{s_{1}}}$ for $1\leq i\leq k_{1}$.
One can see that $J((E(\alpha_{1_{s_{1}}}))^{b}.u)\subsetneqq J(u)$.

On the other hand, if $\alpha_{1_{s_{1}}}(1)c_{1}+\alpha_{1_{s_{1}}}(2)c_{2}=0$,
let
$$\begin{aligned}
\alpha_{0}=\text{min}(\{\alpha_{1_{s_{1}}}-\alpha_{i_{i^{\prime}}},\
\alpha_{i_{i^{\prime}}}\mid1\leq i\leq k,\ 1\leq i^{\prime}\leq s_{i}\}\cap G_{+}).
\end{aligned}$$
We claim that there exists
$\eta\in B(\alpha_{0})$ such that $2\eta\prec \alpha_{1_{s_{1}}}$ and
$\eta(1)c_{1}+\eta(2)c_{2}\neq0$. In fact,
let $0\prec\delta\prec\alpha_{1_{s_{1}}}$ be such that $2\delta\neq\alpha_{1_{s_{1}}}$, and
$$\forall\eta\in B(\text{min}\{\alpha_{0},\ \delta,\ \alpha_{1_{s_{1}}}-\delta\})\setminus
\{\beta\in G_{+}\mid\beta\preceq \alpha_{1_{s_{1}}},\ \beta(1)c_{1}+\beta(2)c_{2}=0\},$$
we have $\eta(1)c_{1}+\eta(2)c_{2}\neq0$, $\eta\prec\delta$ and $\eta\prec\alpha_{1_{s_{1}}}-\delta$.
That is $\eta\prec\alpha_{0}$, $2\eta\prec \alpha_{1_{s_{1}}}$ and $\eta(1)c_{1}+\eta(2)c_{2}\neq0$.
This, together with the assumption that the numbers of the element $\alpha_{1_{s_{1}}}$ occurring in the sequences
$\underline{\a_{i}}$ for $k_{1}< i\leq k$ are all strictly less then $b$, implies that
$$(E(\alpha_{1_{s_{1}}}-\eta))^{b}.\sum\limits_{i=k_{1}+1}^{k}a_{i}t^{-\alpha_{i_{1}}}
t^{-\alpha_{i_{2}}}\cdots t^{-\alpha_{i_{s_{i}}}}v_L=0.$$
Thus
$$(E(\alpha_{1_{s_{1}}}-\eta))^{b}.u=\sum\limits_{i=1}^{k_{1}}b!a_{i}
(\text{det}{\alpha_{1_{s_{1}}} \choose \eta})^{b}(t^{-\eta})^{b}t^{-\alpha_{i_{1}}}
t^{-\alpha_{i_{2}}}\cdots t^{-\alpha_{i_{s_{i}-b}}}v_L\neq0,$$
where $0\prec\alpha_{i_{1}}
\preceq\alpha_{i_{2}}\preceq\cdots\preceq\alpha_{i_{s_{i}-b}}\prec\alpha_{1_{s_{1}}}$ for $1\leq i\leq k_{1}$.
We can see that $J((E(\alpha_{1 n_{1}}-\varepsilon))^{b}.u)\setminus\{\eta\}\subsetneqq J(u)$.
It is clear that, after repeating the above process for finite times, we have
$$0\neq w=\sum\limits_{i=1}^{n}d_{i}t^{-\eta_{i_{1}}}t^{-\eta_{i_{2}}}\cdots t^{-\eta_{i_{s_{i}}}}v_L\in\mathcal{U}(L)u,$$
where $0\neq d_{i}\in\C$, $0\prec\eta_{i_{1}}\preceq\eta_{i_{2}}\preceq\cdots\preceq\eta_{i_{s_{i}}}$
for $1\leq i\leq n$, and $\alpha(1)c_{1}+\alpha(2)c_{2}\neq0$ for all $\alpha\in J(w)$. This implies that we can
turn this case into \textbf{Case $1$}. Therefore, we also obtain that $v_L\in\mathcal{U}(L)u$, and
hence $M(\mathbf{c}, \preceq)$ is irreducible as required.
\end{proof}

Now we consider the situation that the compatible total order $\preceq$ on $G$ is discrete, i.e.,
there exists a smallest positive element $\epsilon$ in $G_{+}$.
For $$u=\sum\limits_{i=1}^{m}\sum\limits_{j=1}^{n}a_{ij}t^{-\alpha_{i_{1}}}
t^{-\alpha_{i_{2}}}\cdots t^{-\alpha_{i_{s_{i}}}}E(-\beta_{j_{1}})
E(-\beta_{j_{2}})\cdots E(-\beta_{j_{k_{j}}})e_{ij}v_L,$$
where $a_{ij}\neq0$, $0\prec\alpha_{i_{1}}
\preceq\alpha_{i_{2}}\preceq\cdots\preceq\alpha_{i_{s_{i}}}$,
$\Z\epsilon\prec\beta_{j_{1}}\preceq\beta_{j_{2}}\preceq\cdots\preceq\beta_{j_{k_{j}}}$, and
$0\neq e_{ij}\in \mathcal{U}(E_{\epsilon})$ for $1\leq i\leq m$, $1\leq j\leq n$, we define
$$l(u)=\text{max}\{k_{j}\mid1\leq j\leq n\}.$$
It is clear that, if $0\neq u\in\mathcal{U}(L^{t}_{-})\mathcal{U}(E_{\epsilon})v_L$, one has $l(u)=0$.
Similarly, for $0\neq w\in\mathcal{U}(L^{t}_{-})v_L$, we write
$$w=\sum\limits_{i=1}^{s}T_{i}t^{-\alpha_{i_{1}}}t^{-\alpha_{i_{2}}}\cdots
t^{-\alpha_{i_{n_{i}}}}v,$$
where $\Z \epsilon\prec\alpha_{i_{1}}\preceq\alpha_{i_{2}}\preceq\cdots\preceq\alpha_{i_{n_{i}}}$
and $0\neq T_{i}\in\mathcal{U}(T_{\epsilon})$  for $1\leq i\leq s$,
we define $$l^{\prime}(w)=\text{max}\{n_{i}\mid 1\leq i\leq s\}.$$
Then one has $l^{\prime}(w)=0$ if $0\neq w\in\mathcal{U}(T_{\epsilon})v_L$.

In order to prove our next main result under the assumption that the total order $\preceq$ on $G$ is discrete, we need the following two lemmas.

\begin{lem}
Suppose that the compatible total order $\preceq$ on $G$ is discrete, then
for any $u\in M(\mathbf{c}, \preceq)$ with $l(u)=m>1$, there exists an element
$0\neq u^{\prime}\in\mathcal{U}(L)u$ such that $l(u^{\prime})=0$.
\end{lem}

\begin{proof}
Write
\[u=\sum_{i=1}^{s}T_{i}E(-\alpha_{i_{1}})
E(-\alpha_{i_{2}})\cdots E(-\alpha_{i_{m}})e_{i}v_L+\sum_{j=1}^{s^{\prime}}T^{\prime}_{j}E(-\beta_{j_{1}})
E(-\beta_{j_{2}})\cdots E(-\beta_{j_{m_{j}}})e^{\prime}_{j}v_L,\]
where $T_{i}, T^{\prime}_{j}\in\mathcal{U}(L^{t}_{-})$, \
$ e_{i}, e^{\prime}_{j}\in \mathcal{U}(E_{\epsilon})$ with $T_i, e_i\not=0$, and
$\Z\epsilon\prec\alpha_{i_{1}}\preceq\alpha_{i_{2}}
\preceq\cdots\preceq\alpha_{i_{m}}$, $\Z\epsilon\prec\beta_{j_{1}}\preceq\beta_{j_{2}}
\preceq\cdots\preceq\beta_{j_{m_{j}}}$, $m>\text{max}\{m_{j}\mid1\leq j\leq s^{\prime}\}$,
$1\leq i\leq s$, $1\leq j\leq s^{\prime}$. Let $\underline{\a_{i}}=
(\alpha_{i_{1}}, \alpha_{i_{2}}, \cdots, \alpha_{i_{m}})$ for $1\leq i\leq s$. One may assume that
$\underline{\a_{1}}\succeq\underline{\a_{2}}\succeq\cdots\succeq\underline{\a_{s}}$. Let $a$ be the number of the element $\alpha_{1_{1}}$ which appears
in the sequence $\underline{\a_{1}}$ for $a$ times.

For $\alpha=\text{min}\{\alpha_{i_{1}}\mid 1\leq i\leq s\}$, we claim that the set
$$\{\alpha-n\epsilon\mid\text{det}{\alpha_{1_{1}} \choose \alpha-n\epsilon}\neq0,\ n\in\N\}$$ is infinite.
In fact, the set $\{\gamma\mid\text{det}{\alpha_{1_{1}} \choose \gamma}=0,\ 0\prec\gamma\preceq\alpha_{1_{1}}\}
\subset\{n\mu(\alpha_{1_{1}})\mid n\in\N\}$ is finite,
which implies that the above set is infinite. Now we choose an integer $n\in\N$ such that
$\text{det}{\alpha_{1_{1}} \choose \alpha-n\epsilon}\neq0$ and $\alpha_{1_{1}}-\alpha+n\epsilon\notin J(u)$ as $J(u)$ is a finite set.

Consider $u_{1}=t^{\alpha-n\epsilon}.u$. Then we have
$$u_{1}=a\text{det}{-\alpha_{1_{1}} \choose \alpha-n\epsilon}T_{1}t^{-\alpha_{1_{1}}+\alpha-n\epsilon}
E(-\alpha_{1_{2}})\cdots E(-\alpha_{1_{m}})e_{1}v_L+A,$$for some $A\in M(\mathbf{c}, \preceq)$.
Since $m>\text{max}\{m_{j}\mid1\leq j\leq s^{\prime}\}$,
$\underline{\a_{1}}\succeq\underline{\a_{2}}\succeq\cdots\succeq\underline{\a_{s}}$
and $\alpha_{1_{1}}-\alpha+n\epsilon\notin J(u)$, we can deduce that
$$a\text{det}{-\alpha_{1_{1}} \choose \alpha-n\epsilon}T_{1}t^{-\alpha_{1_{1}}+\alpha-n\epsilon}
E(-\alpha_{1_{2}})\cdots E(-\alpha_{1_{m}})e_{1}v_L$$ is linearly independent with $A$ if $A\neq0$.
This implies that $0\neq u_{1}\in\mathcal{U}(L)u$ and $l(u_{1})=m-1<l(u)$. Therefore, after repeating the above process for
finitely many times, one obtains that there exists an element $0\neq u^{\prime}\in\mathcal{U}(L)u$ such that $l(u^{\prime})=0$.
\end{proof}

\begin{lem}
Suppose that the compatible total order $\preceq$ on $G$ is discrete
and $\epsilon(1)c_{1}+\epsilon(2)c_{2}\neq0$, Then
for any $w\in M(\mathbf{c}, \preceq)$ with $l(w)=0$, there exists an element
$0\neq w^{\prime}\in\mathcal{U}(L)w\cap\mathcal{U}(L^{t}_{-})v_L$.
\end{lem}

\begin{proof}
Since $l(w)=0$, we have
$$w=\sum_{i=1}^{s}T_{i}E(-i_{1}\epsilon)E(-i_{2}\epsilon)\cdots E(-i_{k_{i}}\epsilon)v_L,$$
where $0\neq T_{i}\in U(L^{t}_{-})$ and $1\leq i_{1}\leq i_{2}\leq\cdots\leq i_{k_{i}}$ for $1\leq i\leq s$.
We may assume that $(s_{1}\epsilon, s_{2}\epsilon, \cdots, s_{k_{s}}\epsilon)$ is maximal  in the set
$\{(i_{1}\epsilon, i_{2}\epsilon, \cdots, i_{k_{i}}\epsilon)\mid 1\leq i\leq s\}$ under the total order $\preceq$. Let $a$ be the number of
the element $s_{k_{s}}\epsilon$ which occurrs in the sequence $(s_{1}\epsilon, s_{2}\epsilon, \cdots, s_{k_{s}}\epsilon)$ for $a$ times.
Consider $w^{\prime}=t^{s_{1}\epsilon}t^{s_{2}\epsilon}\cdots t^{s_{k_{s}}\epsilon}.w$.
Then it is straightforward to check that
$$t^{s_{1}\epsilon}t^{s_{2}\epsilon}\cdots t^{s_{k_{s}}\epsilon}.\sum_{i=1}^{s-1}T_{i}E(-i_{1}\epsilon)
E(-i_{2}\epsilon)\cdots E(-i_{k_{i}}\epsilon)v_L=0$$and
$$
(t^{s_{k_{s}}\epsilon})^{a}.T_{s}E(-s_{1}\epsilon)
E(-s_{2}\epsilon)\cdots E(-s_{k_{s}}\epsilon)v
$$
$$
=a!(s_{k_{s}}\epsilon(1)c_{1}+s_{k_{s}}\epsilon(2)c_{2})^{a}T_{s}
E(-s_{1}\epsilon)E(-s_{2}\epsilon)\cdots E(-(s_{k_{s}-a})\epsilon)v_L.
$$
From these two identities and the assumption $\epsilon(1)c_{1}+\epsilon(2)c_{2}\neq0$, we obtain
$$0\neq w^{\prime}\in\mathcal{U}(L)w\cap\mathcal{U}(L^{t}_{-})v_L$$
as required.
\end{proof}

Now we state the the necessary and sufficient conditions
for the Verma module $M(\mathbf{c}, \preceq)$ to be irreducible under the assumption that the compatible total order $\preceq$ is discrete on $G$.

\begin{thm}
Suppose that the compatible total order $\preceq$ on $G$ is discrete, then the Verma module
$M(\mathbf{c}, \preceq)$ is irreducible if and only if $\epsilon(1)c_{1}+\epsilon(2)c_{2}\neq0$.
\end{thm}

\begin{proof}
If $\epsilon(1)c_{1}+\epsilon(2)c_{2}=0$, it is clear that the $L$-submodule of $M(\mathbf{c}, \preceq)$, generated by
the element $t^{-\epsilon}.v_L$, is a proper submodule of $M(\mathbf{c}, \preceq)$.
Thus $M(\mathbf{c}, \preceq)$ is reducible. Now, to complete the proof of this theorem, we only need to show that
$M(\mathbf{c}, \preceq)$ is irreducible if $\epsilon(1)c_{1}+\epsilon(2)c_{2}\neq0$.
By applying Lemma $3.4$ and $3.5$, we see that, for any $0\neq u\in M(\mathbf{c}, \preceq)$, there exists an element $0\neq u^{\prime}\in\mathcal{U}(L)u\cap\mathcal{U}(L^{t}_{-})v_L$.
Thus without loss of generality, we may assume that $0\neq u\in\mathcal{U}(L^{t}_{-})v_L$.
Write$$u=\sum\limits_{i=1}^{s}T_{i}t^{-\alpha_{i_{1}}}t^{-\alpha_{i_{2}}}\cdots
t^{-\alpha_{i_{n_{i}}}}v_L,$$
where $\Z \epsilon\prec\alpha_{i_{1}}\preceq\alpha_{i_{2}}\preceq\cdots\preceq\alpha_{i_{n_{i}}}$
and $0\neq T_{i}\in\mathcal{U}(T_{\epsilon})$ for $1\leq i\leq s$.
Let
$\underline{\a_{i}}=(\alpha_{i_{1}}, \alpha_{i_{2}}, \cdots, \alpha_{_{i_{n_{i}}}})$ for $1\leq i\leq s$.
We may assume $\underline{\a_{1}}=\text{max}\{\underline{\a_{i}}\mid1\leq i\leq s\}$
and $a$ is the number of the element $\alpha_{_{1_{n_{1}}}}$ which occurs in the sequence $\underline{\a_{1}}$ for $a$ times.
Note that if $l^{\prime}(u)=0$, we have $0\neq u\in\mathcal{U}(T_{\epsilon})v_L$. And
if $l^{\prime}(u)\geq1$, we claim that there exists $0\neq w\in\mathcal{U}(L)u\cap\mathcal{U}(T_{\epsilon})v_L$.
In fact, $\alpha_{1_{n_{1}}}=\text{max}(J(u))$ and $J(u)$ is a finite set,
then $\alpha_{1_{n_{1}}}\succ\Z\epsilon$ and
$\text{det}{\alpha_{_{1_{n_{1}}}} \choose \epsilon}\neq0$. Therefore we can
choose $n\in\N$ such that $n\epsilon\notin J(u)\cap\Z\epsilon$. Consider $E(\alpha_{_{1_{n_{1}}}}-n\epsilon).u$.
Then it is straightforward to check that
$$E(\alpha_{_{1_{n_{1}}}}-n\epsilon).u=a\text{det}{\alpha_{1_{n_{1}}} \choose n\epsilon}T_{1}t^{-\alpha_{1_{1}}}t^{-\alpha_{1_{2}}}\cdots t^{-\alpha_{1_{n_{1}-a}}}t^{-n\epsilon}(t^{-\alpha_{1_{n_{1}}}})^{a-1}v_L+A,$$
and one can deduce that $A$ is linearly independent with $$a\text{det}{\alpha_{1_{n_{1}}} \choose n\epsilon}
T_{1}t^{-\alpha_{1_{1}}}t^{-\alpha_{1_{2}}}\cdots t^{-\alpha_{1_{n_{1}-a}}}t^{-n\epsilon}(t^{-\alpha_{1_{n_{1}}}})^{a-1}v_L$$
if $A\neq0$. And it is clear
that $l^{\prime}(E(\alpha_{_{1_{n_{1}}}}-n\epsilon).u)=l^{\prime}(u)-1$. Repeating the above process, we obtain that there exists an element
$$0\neq w\in\mathcal{U}(L)u\cap\mathcal{U}(T_{\epsilon})v_L.$$
This implies that, for any $0\neq u\in\mathcal{U}(L^{t}_{-})v_L$, there exists an element $w$ such that  $0\neq w\in\mathcal{U}(L)u\cap\mathcal{U}(T_{\epsilon})v_L$.
Let $$w=\sum\limits_{i=1}^{m}a_{i}t^{-i_{1}\epsilon}t^{-i_{2}\epsilon}\cdots t^{-i_{s_{i}}\epsilon}v_L,$$
where $0\neq a_{i}$, $1\leq i_{1}\leq i_{2}\leq\cdots\leq i_{s_{i}}$ for $1\leq i\leq m$. Assume
$$(m_{1}\epsilon, m_{2}\epsilon,\cdots, m_{s_{m}}\epsilon)=
\text{max}\{(m_{1}\epsilon, m_{2}\epsilon,\cdots, m_{s_{m}}\epsilon)\mid1\leq i\leq m\}.$$
Then it is easy to verify that
$$0\neq E(m_{1}\epsilon)E(m_{2}\epsilon)\cdots E(m_{s_{m}}\epsilon).w\in \C v_L,$$
i.e., $v_L\in\mathcal{U}(L)w$.

The above argument implies that
$v_L\in \mathcal{U}(L)u$ for any $0\neq u\in M(\mathbf{c}, \preceq)$, i.e.,
$M(\mathbf{c}, \preceq)$ is irreducible as required.
\end{proof}

Finally, we would like to determine the maximal submodules of the Verma module $M(\mathbf{c}, \preceq)$ for the rank-two Heisenberg-Virasoro algebra $L$ in case $M(\mathbf{c}, \preceq)$ is reducible. We will divide the arguments into two cases. For the case that the compatible total order $\preceq$ on $G$ is dense, we need certain results obtained in [WZ], and for the case that $\preceq$ is discrete, we apply a result from [TWX].

\begin{prop}
Suppose that the compatible total order $\preceq$ is dense, and $(c_{1}, c_{2})=(0, 0)$. Then

(1) If $(c_{3}, c_{4})=(0, 0)$,  the submodule generated by the element $E(-\gamma)v_L$, for
any $\gamma\succ0$, is the unique maximal $\Z^{2}$-graded submodule of the Verma module $M(\mathbf{c}, \preceq)$.

(2) If $(c_{3}, c_{4})\neq(0, 0)$, the submodule generated by the
set $\{t^{-\alpha}v_L\mid \alpha\succ0\}$ is a maximal submodule of the Verma module $M(\mathbf{c}, \preceq)$.
\end{prop}

\begin{proof}
For (1), since the order $\preceq$
is dense on $G$, by Lemma $2.2(1)$ we know that there exists an element $\gamma^{\prime}\succ0$ such that
$\gamma^{\prime}\prec\gamma$ and
$\text{det}{\gamma^{\prime} \choose \gamma}\neq0$. Thus, from the two identities
$$E(-\gamma^{\prime})v_L=(\text{det}{\gamma \choose \gamma^{\prime}})^{-1}
E(\gamma-\gamma^{\prime})E(-\gamma)v_L,$$
and
$$\text{det}{\gamma^{\prime} \choose \gamma}
E(-(\gamma^{\prime}+\gamma))v_L=E(-\gamma)
E(-\gamma^{\prime})v-E(-\gamma^{\prime})E(-\gamma)v_L,$$
we have $E(-\gamma^{\prime})v_L,\ E(-(\gamma^{\prime}+\gamma))v_L\in \mathcal{U}(L)E(-\gamma)v_L$.
Then from the identity
$$\text{det}{\gamma^{\prime} \choose \gamma}
E(-(\gamma^{\prime}+(n+1)\gamma))v_L=E(-\gamma)
E(-(\gamma^{\prime}+n\gamma))v_L-E(-(\gamma^{\prime}+n\gamma))
E(-\gamma)v_L,$$ and, by induction on $n\in \Z_+$, one sees that
$E(-(\gamma^{\prime}+n\gamma))v_L\in \mathcal{U}(L)E(-\gamma)v_L$ for any
$n\in \Z_{+}$. By Lemma $2.2(2)$, we know that, for any $\b\in G_{+}$, there exists $n\in \Z_+$ such that
$\gamma^{\prime}+(n+1)\gamma\succ\gamma^{\prime}+n\gamma\succ\b$
for some $n\in \Z_{+}$. Hence one of the two
determinants $\text{det}{\gamma^{\prime}+n\gamma \choose\b}$ and
$\text{det}{\gamma^{\prime}+(n+1)\gamma \choose\b}$ is not equal to zero, say
$\text{det}{\gamma^{\prime}+n\gamma \choose\b}\neq0$. Note that
$$\text{det}{-\gamma^{\prime}-n\gamma \choose -\b}E(-\b)v_L=E(-\b+\gamma^{\prime}
+n\gamma)E(-(\gamma^{\prime}+n\gamma))v_L.$$
Thus we have $E(-\b)v_L\in \mathcal{U}(L)E(-\gamma)v_L$ for any $\b\in G_{+}$. By assumption, for any
$\d\succ0$, there exists $\alpha\in G_{+}$ such that $\d\prec\alpha\prec 2\d$ and
$\text{det}{\d \choose \alpha}\neq0$. Hence
$$t^{\alpha}E(-\d-\alpha)v_L=\text{det}{-\d \choose \alpha}t^{-\d}v_L,$$
which implies that $t^{-\d}v_L\in \mathcal{U}(L)E(-\gamma)v_L$ for any $\d\in G_{+}$.
Moreover, since $(c_{3}, c_{4})=(0,0)$, it is easy to check that
$$\mathcal{U}(L)E(-\gamma)v_L=\sum\limits_{\b\in G_{-}}M(\mathbf{c}, \preceq)_{\b}.$$ Therefore, the submodule
 generated by $E(-\gamma)v_L$ is the unique maximal $\Z^{2}$-graded submodule of $M(\mathbf{c}, \preceq)$. This finishes the proof of the first part of the proposition.

For the second statement (2), since $(c_{3}, c_{4})\neq(0, 0)$, it is clear that the submodule
 generated by the set $\{t^{-\alpha}v_L\mid \alpha\succ0\}$ is
$\mathcal{U}(L^{E}_{-})\mathcal{U}(L^{t}_{-})L^{t}_{-}v_L$. By applying Theorem $3.1$ in the paper \cite{WZ}, we  see
that $\mathcal{U}(L^{E}_{-})\mathcal{U}(L^{t}_{-})L^{t}_{-}v_L$ is a maximal submodule of $M(\mathbf{c}, \preceq)$, i.e.,
the submodule generated by the
set $\{t^{-\alpha}v_L\mid \alpha\succ0\}$ is  maximal as required.
\end{proof}

Now we assume that the order is discrete, and before state our result we need to introduce some more symbols.
Since $\epsilon,\epsilon^\prime$ form a $\Z$-basis of $G$, then for any $\alpha\in G$, we have
$\alpha=\alpha[1]\epsilon+\alpha[2]\epsilon^{\prime}$, for some $\alpha[1],\alpha[2]\in\Z$. Set $E=\sum\limits_{k\in \Z\backslash\{0\}}\C E(k\epsilon)$,
$H_{\pm}=\sum\limits_{\pm k\in \N}(\C t^{k\epsilon}\oplus \C E(k\epsilon))$, and $H=H_{+}\oplus L_{0}\oplus H_{-}$. For the sake of convenience, we denote by $X(\alpha)$ the element $t^{\alpha}$ or $E(\alpha)$ for $\alpha\in G_{+}$.  Let $$\begin{aligned}
B=\{1,t^{-\beta_{1}}\cdots t^{-\beta_{m}},E(-\alpha_{1})\cdots E(-\alpha_{n}),
t^{-\beta_{1}}\cdots t^{-\beta_{m}}E(-\alpha_{1})\cdots E(-\alpha_{n})\mid \\
m,n\in\N,\ 0\preceq\alpha_{n}\preceq\cdots\preceq\alpha_{1},\ 0\preceq\beta_{m}\preceq\cdots\preceq\beta_{1}\}
\end{aligned}$$be a basis of $\mathcal{U}(L_{-})$, and $M(H)=\mathcal{U}(H)v_L$. For $h\in\N$, we set
$$B_{H}(h)=\{X(-\alpha_{1})\cdots X(-\alpha_{n})\in B\mid n\in\N,
\alpha_{i}\in G_{+}\backslash \Z\epsilon, \sum_{i=1}^{n}\alpha_{i}[2]=h\},$$
and set $B_{H}(0)=\{1\}$.
Let $M(h)=B_{H}(h)M(H)$ for $h\in \Z_{+}$. For any pair
$(a, b)=((a_{i})_{i\in \N}, (b_{i})_{i\in \N})\in \C^{\N}\times\C^{\N}$,
let $I_{ab}$ denote the ideal of $\mathcal{U}(H_{-})$ generated by $\{t^{-i\epsilon}-a_{i}, E(-i\epsilon)-b_{i}\mid i\in\N\}$.
For any $\xi=(\xi_{i})_{i\in\N}\in\C^{\N}$, let $J_{\xi}$ denote the ideal of $\mathcal{U}(T_{\epsilon})$
generated by $\{t^{-i\epsilon}-\xi_{i}\mid i\in\N\}$. Then we have the following
statement.

\begin{prop}
Suppose that the compatible total order $\preceq$ is discrete on $G$, and assume that $\epsilon(1)c_{1}+\epsilon(2)c_{2}=0$. Then

(1) If $\epsilon(1)c_{3}+\epsilon(2)c_{4}=0$,  $\{M_{ab}\}$ is the set of all
maximal $\Z^{2}$-graded submodules of the Verma module $M(\mathbf{c}, \preceq)$, where
$(a, b)=((a_{i})_{i\in \N}, (b_{i})_{i\in \N})\in
\C^{\N}\times\C^{\N}$, and $M_{ab}=\sum\limits_{h\in \Z_{+}}M_{ab}(h)$ with
$M_{ab}(h)=\{u\in M(h)\mid X(\epsilon^{\prime}+i_{1}\epsilon)\cdots X(\epsilon^{\prime}+i_{h}\epsilon)u\in I_{ab}v_L,\
\forall i_{1}, \cdots, i_{h}\in \Z\}$.

(2) If $\epsilon(1)c_{3}+\epsilon(2)c_{4}\neq0$, $\{M_{\xi}\mid \xi\in\C^{\N}\}$ is the set of all
maximal $\Z^{2}$-graded submodules of the Verma module $M(\mathbf{c}, \preceq)$, where $M_{\xi}=\sum\limits_{h\in\Z_{+}}M_{\xi}(h)$ with
$M_{\xi}(h)=\{u\in M(h)\mid X(\epsilon^{\prime}+i_{1}\epsilon)\cdots X(\epsilon^{\prime}+i_{h}\epsilon)u\in \mathcal{U}(E)J_{\xi}v_L,\
\forall i_{1}, \cdots, i_{h}\in \Z\}$.
\end{prop}

\begin{proof}
We take $\varphi=0$ in the definition of the universal Whittaker module $M_{\varphi, k_{1}, k_{2}, k_{3}, k_{4}}$ defined in \cite{TWX}, then
the universal Whittaker modules are exactly the Verma modules for the rank-two Heisenberg-Virasoro algebra $L$. Therefore the statement of this proposition follows from Proposition $4.3$ and Proposition $4.5$ in \cite{TWX}.
\end{proof}

\end{document}